\documentclass{birkjour}

 \newtheorem{thm}{Theorem}[section]
 \newtheorem{cor}[thm]{Corollary}
 \newtheorem{lem}[thm]{Lemma}
 \newtheorem{prop}[thm]{Proposition}
 \theoremstyle{definition}
 \newtheorem{defn}[thm]{Definition}
 \newtheorem{rem}[thm]{Remark}
 \newtheorem{ex}[thm]{Example}
 \numberwithin{equation}{section}

\usepackage{amsfonts,amsmath,amssymb}

\usepackage[colorlinks]{hyperref}

\usepackage{enumerate}

\usepackage{tikz}

\usepackage{tikz-cd}

\numberwithin{equation}{section}

\newcommand{\cl}[1]{\mathcal{#1}} 

\newcommand{\bb}[1]{\mathbb{#1}}

\newcommand{\nor}[1]{\left\Vert #1\right\Vert}

\DeclareMathOperator*{\tens}{\otimes}

\begin{document}

\title[Similarities for the maximal tensor product of certain C*-algebras]
 {Similarities for the maximal tensor product of certain C*-algebras}
\vspace{3em}

\author{E. Papapetros}
\address{University of the Aegean\\Faculty of Sciences\\ Department of Mathematics\\832 00 Karlovasi, Samos, Greece}
\email{e.papapetros@aegean.gr}

\subjclass{Primary: 47L30; Secondary: 46L05, 46L10, 47L55}

\keywords{C*-algebras, von Neumann algebras, similarity problem}

\begin{abstract}
We prove that if the unital $C^*$-algebras $\cl A$ and $\cl B$ satisfy Kadison's similarity property and the length $L=L\left(\cl A\tens \limits_{max}\cl B\right)$ of their maximal tensor product is finite, then $\,\cl A\tens\limits_{max}\cl B$ satisfies Kadison's similarity property with similarity length $\ell\left(\cl A\tens\limits_{max}\cl B\right)\leq L \max\left\{\ell(\cl A),\,\ell(\cl B)\right\}.$
\end{abstract}

\maketitle

\section{Introduction}
Given a $C^*$-algebra $ \mathcal A$ and a bounded homomorphism $\rho\colon \mathcal A\rightarrow \mathcal B(H),$ where $H$ is a Hilbert space, does there exist 
an invertible operator $S\in \mathcal B(H)$ such that $\pi (\cdot)=S^{-1}\,\rho(\cdot)\,S$ defines a $*$-homomorphism of $\cl A$? If this happens, we say that $\rho$ is similar to $\pi.$ We say 
that $\mathcal A$ satisfies the similarity property ((SP)) if every bounded homomorphism $\rho\colon \mathcal A\rightarrow \mathcal B(H),$ where $H$ is a Hilbert space, 
is similar to a $*$-homomorphism.

The above property was introduced by Kadison in \cite{kad} where he conjectured that all $C^*$-algebras satisfy the similarity property. Kadison's similarity problem is equivalent to a number of questions, including the problem of hyperreflexivity of all von Neumann algebras, the derivation problem, the invariant operator range problem and the problem of finite length of $C^*$-algebras \cite{pisier1}. The author and Eleftherakis in their joint work \cite{ele-pap}, proved that Kadison's similarity problem is equivalent to the following property, which they named {\bf (CHH):}

{\em Every hyperreflexive separably acting von Neumann algebra is completely hyperreflexive}.

Recently, in \cite{pap} the author proved that Kadison's similarity problem is also equivalent to the following property {\bf (EP):}

{\em Every separably acting von Neumann algebra with a cyclic vector is hyperreflexive}.

We say that a von Neumann algebra $\cl M$ satisfies the weak similarity property ((WSP)) if every $\rm{w}^*$-continuous, unital and bounded homomorphism $u\colon \cl M\rightarrow \cl B(H),$ where $H$ is a Hilbert space, is similar to a $*$-homomorphism. 
The connection between (SP) and (WSP) is given by the following lemma.

\begin{lem}\cite{ele-pap}
    \label{SP-WSP}
   Let $\cl A$ be a unital $C^*$-algebra. The algebra $\cl A$ satisfies (SP) if and only if its second dual $\cl A^{**}$ satisfies (WSP). 
\end{lem}

Pisier in \cite{pisier3} proved that a $C^*$-algebra $\cl A$ satisfies (SP) if and only if its length $\ell(\cl A)$ is finite. 

\begin{defn}\cite{pisier3}
   The length of a $C^*$-algebra $\cl A$ is the smallest $\ell(\cl A)\geq 1$ for which there is a constant $C=C(\cl A)>0$ such that every bounded homomorphism $\pi\colon \cl A\to \cl B(H),$ where $H$ is a Hilbert space, satisfies $$||\pi||_{cb}\leq C\,||\pi||^{\ell(\cl A)}.$$  
\end{defn}
We list below the known cases of $C^*$-algebras with finite length.
\begin{enumerate}[(i)]
    \item If $\cl A$ is a nuclear $C^*$-algebra, then $\ell(\cl A)\leq 2$ \cite{pn}.
    \item If $\cl A$ is a $C^*$-algebra with no tracial states, then $\ell(\cl A)\leq 3$ \cite{pis}.
    \item If $\cl M$ is a type $II_1$ factor with property $\Gamma,$ then $\ell(\cl M)=3$ \cite{chris1}.
    \item For any $C^*$-algebra $\cl A$ it holds that $\ell\left(\cl A \tens \limits_{min} \cl K\right)\leq 3,$ where $\cl K$ is the $C^*$-algebra of compact operators on $\ell^2(\bb N)$ \cite{haag, pis}.
\end{enumerate}
It is not known whether there exists a unital $C^*$-algebra $\cl A$ with $\ell(\cl A)>3.$ An affirmative answer to Kadison's similarity problem would imply the existence of a positive integer $\ell_0$ such that $\ell(\cl A)\leq \ell_0$ for every $C^*$-algebra $\cl A$ \cite[Proposition 15]{pis}.

In \cite{ele-pap} the author and Eleftherakis proved that the normal spatial tensor product of an injective von Neumann algebra and a von Neumann algebra satisfying (WSP), satisfies (WSP). By combining Lemma \ref{SP-WSP} and the fact that a unital $C^*$-algebra is nuclear if and only if its second dual is injective \cite[Theorem 6.4]{ef-lan}, the following question occurs naturally: \\
   (1) {\em Does the minimal tensor product of a nuclear unital $C^*$-algebra with a unital $C^*$-algebra satisfying (SP), also satisfy (SP)} ?
 
 More generally :\\   
(2) {\em Does the maximal tensor product of two unital $C^*$-algebras satisfying (SP), also satisfy (SP)} ?

In Section 2, we set the stage by presenting definitions and lemmas that are useful for the next section. Moreover, we prove that if there exist unital $C^*$-algebras $\cl A$ and $\cl B$ such that $d_1(\cl A,\,\cl B)<\infty$ whereas $d(\cl A,\,\cl B)=\infty,$ as these quantities defined by Pisier in \cite{pairs}, then the $C^*$-algebra $\cl A\tens\limits_{max}\cl B$ does not satisfy (SP). 

In Section 3 we prove the main theorem of the paper. Namely, in Theorem \ref{main} we prove that if the unital $C^*$-algebras $\cl A$ and $\cl B$ satisfy (SP) and the length $L$ of their maximal tensor product is finite, then $\cl A\tens\limits_{max}\cl B$ satisfies (SP) and its similarity length $\ell\left(\cl A\tens\limits_{max}\cl B\right)$ is less than $L \max\left\{\ell(\cl A),\,\ell(\cl B)\right\}.$ This fact answers partially question (2). In \cite{pairs}, Pisier proved that if $\cl A$ is a unital and nuclear $C^*$-algebra and $\cl B$ is an arbitrary unital $C^*$-algebra, then $L\left(\cl A\tens\limits_{max}\cl B\right)\leq 3,$ and so question (1) above has an affirmative answer in the unital case. We generalise this result to the non-unital case as well.

\vspace{1em} 
In what follows, $\cl B(H)$ is the $C^*$-algebra of all linear and bounded operators from the Hilbert space $H$ to itself. For two $C^*$-algebras $\cl A$ and $\cl B,$ $\,\cl A\tens \limits_{min} \cl B$ denotes their minimal (spatial) tensor product whereas $\cl A\tens\limits_{max} \cl B$ denotes their maximal tensor product. For further details see \cite{bm, pisbook}. 

We remind the reader the notions of a completely bounded map and a nuclear $C^*$-algebra that will be used in the following sections.

\begin{defn}
  Let $u\colon \mathcal X\to \mathcal Y$ be a linear map between operator spaces. For each $n\in\mathbb N$ consider the matrix amplification $$u_n\colon M_n(\mathcal X)\to M_n(\mathcal Y),\,\, u_n((X_{i,j}))=(u(X_{i,j})).$$
 
If $\,\sup_{n} ||u_n||<\infty, $  then we say that $u$ is completely bounded with cb norm $$||u||_{cb}:=\sup_{n} ||u_n||.$$   
\end{defn}

\begin{defn}
    A $C^*$-algebra $\cl A$ is called nuclear if $\cl A\tens \limits_{min}\cl B\cong \cl A\tens \limits_{max} \cl B$ for every $C^*$-algebra $\cl B.$
\end{defn}
For further details on nuclear $C^*$-algebras and their properties we refer the reader to \cite{bo, ef-choi, ef-lan, la}.

\section{Preliminaries}
We begin this section with the definition of the length for a pair of $C^*$-subalgebras \cite{pairs}.

Let $\cl A$ and $\cl B$ be $C^*$-subalgebras of a $C^*$-algebra $\cl Z.$ For every $n\in\bb N$ we will view $M_n(\cl A)$ and $M_n(\cl B)$ as $C^*$-subalgebras of $M_n(\cl Z),$ so that if $x_1\in M_n(\cl A)$ and $x_2\in M_n(\cl B),$ then the product $x_1 x_2$ belongs to $M_n(\cl Z),$ and similarly for a product of rectangular matrices. 

Let $d\geq 1$ be an integer. We say that $L(\cl Z;\,\cl A,\,\cl B)\leq d$ if there is a constant $C$ such that for any $n\in\bb N$ and any $x\in M_n(\cl Z)$ with $||x||_{M_n(\cl Z)}<1$ and for any $\epsilon>0,$ there exists $N\in\bb N$ for which we can find matrices $x_1,\,x_2,\,...,x_d$ and $y_1,\,y_2,\,...,y_d,$ with entries either all in $\cl A$ or all in $\cl B,$ where $x_1,\,x_2,\,...,x_d$ are of sizes respectively, $n\times N,\,N\times N,\,...,N\times N$ and $N\times n,$ and similarly for $y_1,\,y_2,\,...,y_d,$ satisfying $$\prod_{j=1}^{d}||x_j||+\prod_{j=1}^{d}||y_j||<C$$ and finally $$\nor{x-\prod_{j=1}^{d}x_j-\prod_{j=1}^{d}y_j}<\epsilon.$$

In what follows, let $\cl A$ and $\cl B$ be unital $C^*$-algebras.  We identify $\cl A$ with $\cl A\,\otimes 1_{\cl B}$ and $\cl B$ with $1_{\cl A}\,\otimes \cl B$ and we view them as unital $C^*$-subalgebras of the unital $C^*$-algebra $\cl Z=\cl A\tens\limits_{max}\cl B.$

\begin{defn}\cite{pairs}
    We define the length of $\cl A\tens\limits_{max}\cl B$ to be $$L\left(\cl A\tens\limits_{max}\cl B\right):=L\left(\cl A\tens\limits_{max}\cl B;\, \cl A\,\otimes 1_{\cl B},\,1_{\cl A}\,\otimes \cl B\right).$$
\end{defn}

By Gelfand-Naimark theorem, there exist Hilbert spaces $H_0$ and $K_0$ as well as faithful unital $*$-representations $\pi\colon \cl A\to \cl B(H_0)$ and $\rho\colon \cl B\to \cl B(K_0).$ Let $H=H_0\otimes_{2} K_0$ be the Hilbert tensor product of $H_0$ and $K_0.$ The maps $u$ and $v$ defined by $$u\colon \cl A\to \cl B(H),\,\,u(a)=\pi(a)\otimes I_{K_0}$$ $$v\colon \cl B\to \cl B(H),\,\,v(b)=I_{H_0}\otimes \rho(b)$$ are unital $*$-homomorphisms (so they are unital and completely bounded) with commuting ranges since for all $a\in\cl A$ and $b
\in\cl B$ we have $$u(a) v(b)=\pi(a)\otimes \rho(b)=v(b) u(a).$$ 
Therefore, the following definitions are meaningful.

\begin{defn}\cite{pairs}
   We say that $\,d_1(\cl A,\,\cl B)\leq d$ if there is a constant $M>0$ such that for any pair $u\colon \cl A\to \cl B(H),\,v\colon \cl B\to \cl B(H)$ of completely bounded unital homomorphisms with commuting ranges, the homomorphism $$u v\colon \cl A\otimes \cl B\to \cl B(H),\,\,a\otimes b\mapsto u(a) v(b)$$ is bounded on $\cl A\tens\limits_{max}\cl B$ with $||u v||\leq M \max\left\{||u||_{cb},\,||v||_{cb}\right\}^d.$ 
\end{defn}

\begin{defn}\cite{pairs}
     We say that $d(\cl A,\,\cl B)\leq d$ if there is a constant $M>0$ such that for any pair $u\colon \cl A\to \cl B(H),\,v\colon \cl B\to \cl B(H)$ of completely bounded unital homomorphisms with commuting ranges, the homomorphism $u v$ is completely bounded on $\cl A\tens\limits_{max}\cl B$ with $||u v||_{cb}\leq M \max\left\{||u||_{cb},\,||v||_{cb}\right\}^d.$  
\end{defn}

The number $d(\cl A,\,\cl B)$ (resp. $\,d_1(\cl A,\,\cl B)$) is defined as the infimum of the numbers $d\geq 1$ such that the above property holds. We set $d(\cl A,\,\cl B)=\infty$ (resp. $\,d_1(\cl A,\,\cl B)=\infty$) if there exist a Hilbert space $H_0$ as well as completely bounded unital homomorphisms $u_0\colon \cl A\to \cl B(H_0),\,v_0\colon \cl B\to \cl B(H_0)$ with commuting ranges such that the homomorphism $u_0 v_0$ is not completely bounded (resp. not bounded) on $\cl A\tens\limits_{max}\cl B.$ Clearly, $d_1(\cl A,\,\cl B)=d_1(\cl B,\,\cl A)$ and $d(\cl A,\,\cl B)=d(\cl B,\,\cl A).$

\begin{prop}
\label{ineq}
$d_1(\cl A,\,\cl B)\leq d(\cl A,\,\cl B)$.
\end{prop}

\begin{proof}
  Obviously, $d_1(\cl A,\,\cl B)=\infty$ implies $d(\cl A,\,\cl B)=\infty.$ Assume that $d(\cl A,\,\cl B)=d<\infty.$ There exists $M>0$ such that for any pair $u\colon \cl A\to \cl B(H),\,v\colon\cl B\to \cl B(H)$ of completely bounded unital homomorphisms with commuting ranges, the homomorphism $u v$ is completely bounded on $\cl A\tens\limits_{max}\cl B$ with $$||u v||_{cb}\leq M\,\max\left\{||u||_{cb},||v||_{cb}\right\}^d,$$ which implies that for the same constant $M>0$ and for any pair $(u,v)$ as above, the homomorphism $u v$ is bounded with norm $$||u v||\leq ||u v||_{cb}\leq M \max\left\{||u||_{cb},||v||_{cb}\right\}^d.$$ Therefore, $d_1(\cl A,\,\cl B)\leq d=d(\cl A,\,\cl B).$
\end{proof}

In \cite[Remark 25]{pairs}, Pisier proved that under the hypothesis that for any pair $u\colon \cl A\to \cl B(H),\,v\colon \cl B\to \cl B(H)$ of completely bounded unital homomorphisms with commuting ranges, the homomorphism $u v$ is completely bounded on $\cl A\tens\limits_{max}\cl B,$ then $d(\cl A,\,\cl B)<\infty$ and, moreover, $d(\cl A,\,\cl B)=L\left(\cl A\tens\limits_{max}\cl B\right).$

Assuming now that for any pair $u\colon \cl A\to \cl B(H),\,v\colon \cl B\to \cl B(H)$ of completely bounded unital homomorphisms with commuting ranges, the homomorphism $u v$ is bounded on $\cl A\tens\limits_{max}\cl B,$ then Proposition \ref{ineq} and Lemma 22 in \cite{pairs} imply that $$d_1(\cl A,\,\cl B)\leq d(\cl A,\,\cl B)\leq d,$$ where $d$ is the minimun integer having the property $L\left(\cl A\tens\limits_{max}\cl B\right)\leq d.$
However, according to \cite[Remark 25]{pairs}, if we only assume that $L\left(\cl A\tens\limits_{max}\cl B\right)<\infty,$ we cannot verify in full generality that $d(\cl A,\,\cl B)\leq L\left(\cl A\tens\limits_{max}\cl B\right).$

\begin{prop}
    Let $d_1=d_1(\cl A,\cl B)<\infty.$ If the $C^*$-algebra $\cl A\tens\limits_{max}\cl B$ satisfies (SP), then $$L\left(\cl A\tens\limits_{max}\cl B\right)=d(\cl A,\cl B)\leq \ell\,d_1(\cl A,\,\cl B),$$ where $\ell$ is the finite length of the $C^*$-algebra $\cl A\tens\limits_{max}\cl B.$
\end{prop}

\begin{proof}
    Let $u\colon \cl A\to \cl B(H),\,v\colon \cl B\to \cl B(H)$ be completely bounded unital homomorphisms with commuting ranges. According to our hypothesis, the homomorphism $u v$ is bounded on $\cl A\tens\limits_{max}\cl B$ with norm \begin{equation} \label{d_1}
        ||u v||\leq M \max\left\{||u||_{cb},||v||_{cb}\right\}^{d_1}.
    \end{equation}
 Since $\cl A\tens\limits_{max}\cl B$ satisfies (SP) it follows that $u v$ is similar to a $*$-homomorphism, which implies that $u v$ is completely bounded on $\cl A\tens\limits_{max}\cl B$ and from \cite{pisier3} its cb norm satisfies the inequality \begin{equation}\label{length1}
     ||u v||_{cb}\leq \tilde{M} ||u v||^{\ell}.
 \end{equation}
By combining (\ref{d_1}) and (\ref{length1}) we have $$||u v||_{cb}\leq \tilde{M} M^{\ell} \max\left\{||u||_{cb},||v||_{cb}\right\}^{\ell d_1}.$$
Therefore, $d(\cl A,\,\cl B)\leq \ell d_1<\infty$ and by Theorem 24 in \cite{pairs}, we deduce that $L\left(\cl A\tens\limits_{max}\cl B\right)=d(\cl A,\cl B)\leq \ell d_1.$
\end{proof}


\section{The main theorem}
In this section we prove the main theorem of the paper, namely:

\begin{thm}
\label{main}
  Let $\cl A$ and $\cl B$ be unital $C^*$-algebras satisfying (SP) and such that $L=L\left(\cl A\tens\limits_{max}\cl B\right)<\infty.$ Then $\,\cl A\tens\limits_{max}\cl B$ satisfies (SP) with similarity length $\ell\left(\cl A\tens\limits_{max}\cl B\right)\leq L \max\left\{\ell(\cl A),\,\ell(\cl B)\right\}.$
\end{thm}

\begin{proof}
We may assume that $\ell(\cl A)\geq \ell(\cl B).$ Let $\pi\colon \cl A\tens\limits_{max}\cl B\to \cl B(H)$ be a unital bounded homomorphism. The maps $$\pi_1\colon \cl A\to \cl B(H),\,\,\pi_1(a)=\pi(a\otimes 1_{\cl B})$$ $$\pi_2\colon \cl B\to \cl B(H),\,\,\pi_2(b)=\pi(1_{\cl A}\otimes b)$$ are unital bounded homomorphisms of $\cl A$ and $\cl B,$ respectively, into $\cl B(H)$ with commuting ranges. Since $\cl A$ and $\cl B$ satisfy (SP), $\,\pi_1$ and $\pi_2$ are similar to $*$-homomorphisms, and thus $\pi_1$ and $\pi_2$ are completely bounded. Proposition 12 in \cite{length} implies the existence of an invertible operator $S\in\cl B(H)$ such that both homomorphisms $$x\mapsto \rho_{j}(x)=S\,\pi_{j}(x)\,S^{-1},\,\,j=1,2$$ are completely contractive on $\cl A$ and $\cl B,$ respectively, with commuting ranges. By the fact that $\cl A$ and $\cl B$ are $C^*$-algebras, $\rho_1$ and $\rho_2$ are $*$-homomorphisms, and thus $\rho\colon \cl A\tens\limits_{max}\cl B\to \cl B(H)$ given by $\rho(a\otimes b)=\rho_1(a)\,\rho_2(b),\,\,a\in\cl A,\,b\in\cl B$ is a $*$-homomorphism similar to $\pi$ since for all $a\in\cl A,\,b\in\cl B$ we have  $$S\,\pi(a\otimes b)\,S^{-1}=(S \pi_1(a) S^{-1})(S \pi_2(a) S^{-1})=\rho_1(a) \rho_2(b)=\rho(a\otimes b).$$
Moreover, by \cite[Proposition 12]{length} it holds that \begin{equation}
    \label{1st} ||\pi||_{cb}\leq K \max\left\{||\pi_1||_{cb},\,||\pi_2||_{cb}\right\}^{L}.
\end{equation}
Finally, from \cite{pisier3} we have \begin{equation}\label{similarity length}
     ||\pi_1||_{cb}\leq c_1(\cl A)\,||\pi_1||^{\ell(\cl A)}\leq c_1(\cl A)\, ||\pi||^{\ell(\cl A)}.
 \end{equation}
 \begin{equation}\label{similarity length-1}
     ||\pi_2||_{cb}\leq c_2(\cl B)\,||\pi_2||^{\ell(\cl B)}\leq c_2(\cl B)\,||\pi||^{\ell(\cl A)}.
 \end{equation}
By combining (\ref{1st}), (\ref{similarity length}) and (\ref{similarity length-1}) we obtain $$
    ||\pi||_{cb}\leq \tilde{K}\,||\pi||^{L\,\ell(\cl A)},$$ as desired.
\end{proof}

\begin{rem}
    The converse of the above Theorem is not true. Indeed, the $C^*$-algebra $\cl B(H)\tens\limits_{max}\cl B(H)$ satisfies (SP) (as a $C^*$-algebra without tracial states). However, by \cite[Corollary 29]{pairs} we have $L\left(\cl B(H)\tens\limits_{max}\cl B(H)\right)=\infty.$ This example shows that compression strategy cannot work through the use of similarity degree.
\end{rem}

\begin{prop}
    Let $\cl A$ and $\cl B$ be unital $C^*$-algebras such that $\cl A$ is nuclear. If $\pi\colon \cl A\tens\limits_{min}\cl B\to \cl B(H)$ is a unital bounded homomorphism such that $\pi|_{\cl B}$ is completely bounded, then $\pi$ is completely bounded with $$||\pi||_{cb}\leq C \max\left\{||\pi|_{\cl A}||_{cb},\,||\pi|_{\cl B}||_{cb}\right\}^3.$$
\end{prop}

\begin{proof}
 The map $\pi|_{\cl A}\colon \cl A\to \cl B(H),\,\,\pi(a)=\pi(a\otimes 1_{\cl B})$ is a unital bounded homomorphism of $\cl A$ into $\cl B(H).$ The $C^*$-algebra $\cl A$ satsifies (SP) as a nuclear one, and thus $\pi|_{\cl A}$ is completely bounded. Since $L\left(\cl A\tens\limits_{min}\cl B\right)\leq 3$ \cite{pairs}, it follows from Lemma 22 in \cite{pairs} that $\pi$ is completely bounded with cb-norm $||\pi||_{cb}\leq C \max\left\{||\pi|_{\cl A}||_{cb},\,||\pi|_{\cl B}||_{cb}\right\}^3.$ 
\end{proof}

The above Proposition generalises Corollary 2.3 in \cite{pop} and has a lot of other corollaries as well. We present them below:

\begin{cor}
\label{sp-nuc}
The tensor product of a nuclear unital $C^*$-algebra $\cl A$ with a unital $C^*$-algebra $\cl B$ satisfying (SP), also satisfies (SP).
\end{cor}

We strengthen Corollary \ref{sp-nuc} by proving the same result for the non-unital case. Before this we prove a useful lemma. 

\begin{lem}
    \label{equal}
    Let $\cl A$ and $\cl B$ be isomorphic $C^*$-algebras. Then $\ell(\cl A)=\ell(\cl B).$    
\end{lem}

\begin{proof}
    Let $f\colon \cl A\to \cl B$ be a $*$-isomorphism between $\cl A$ and $\cl B.$ If $\ell(\cl B)=\infty$ then by \cite{pisier3} the $C^*$-algebra $\,\cl B$ does not satisfy (SP), and thus $\cl A$ does not satisfy (SP) which implies that $\ell(\cl A)=\ell(\cl B)=\infty.$ Assume that $\ell(\cl B)<\infty.$ Since the relation $\cong$ of isomorphism between $C^*$-algebras is symmetric, it suffices to prove that $\ell(\cl A)\leq \ell(\cl B).$ To this end, let $\pi\colon \cl A\to \cl B(H)$ be a bounded representation of $\cl A$ on the Hilbert space $H.$ The map $\tilde{\pi}=\pi\circ f^{-1}\colon \cl B\to \cl B(H)$ is a bounded representation of $\cl B,$ and thus $||\tilde{\pi}||_{cb}\leq c(\cl B)\,||\tilde{\pi}||^{\ell(\cl B)}.$ Since $\pi=\tilde{\pi}\circ f$ and $f$ is completely contractive, we have $$||\pi||_{cb}\leq ||\tilde{\pi}||_{cb}\,||f||_{cb}\leq c(\cl B)\,||\tilde{\pi}||^{\ell(\cl B)}\leq c(\cl B)\,||\pi||^{\ell(\cl B)}.$$ Therefore, $\ell(\cl A)\leq \ell(\cl B).$
\end{proof}

We are treating now the non-unital version of Corollary \ref{sp-nuc}.
\begin{prop}
\label{gen}
    Let $\cl A$ and $\cl B$ be $C^*$-algebras such that $\cl A$ satisfies (SP) and $\cl B$ is nuclear. Then $\cl A\tens\limits_{min} \cl B$ satisfies (SP) and $\ell\left(\cl A\tens\limits_{min}\cl B\right)\leq 3 \max\left\{\ell(\cl A),\,\ell(\cl B)\right\}.$ 
\end{prop}

\begin{proof}
Assume that $\cl A$ and $\cl B$ act on the Hilbert spaces $H$ and $K,$ respectively. The minimal tensor product $\cl A\tens\limits_{min} \cl B$ of $\cl A$ and $\cl B$ is a closed, self-adjoint two-sided ideal of the minimal tensor product $\cl A^1\tens\limits_{min} \cl B^1$ of the unitizations $\cl A^1$ and $\cl B^1$ of the $C^*$-algebras $\cl A$ and $\cl B,$ respectively, into $\cl B(H\otimes K).$ The $C^*$-algebra $\cl A^1$ satisfies (SP). Indeed, $\cl A$ is a two-sided ideal of $\cl A^1$ such that $\cl A^1/\cl A\cong \bb C$ and by combining \cite[Remark 6]{pisier4} and Lemma \ref{equal} we have $$\ell(\cl A^1)=\max\left\{\ell(\cl A),\,\ell(\cl A^1/\cl A)\right\}=\max\left\{\ell(\cl A),\,\ell(\bb C)\right\}=\ell(\cl A)<\infty.$$ Moreover, Corollary 8.17 in \cite{pistensor} yields that $\cl B^1$ is nuclear since $\cl B$ is nuclear and $\cl B^1/\cl B\cong \bb C$ is also nuclear. According to Corollary \ref{sp-nuc}, the $C^*$-algebra $\,\cl A^1\tens\limits_{min} \cl B^1$ satisfies (SP). By \cite[Remark 6]{pisier4}, the fact that $L\left(\cl A^1\tens\limits_{min}\cl B^1\right)\leq 3$ \cite{pairs} and Theorem \ref{main}, we have $$\ell\left(\cl A\tens\limits_{min} \cl B\right)\leq \ell\left(\cl A^1\tens\limits_{min} \cl B^1\right)\leq 3 \max\left\{\ell(\cl A^1),\,\ell(\cl B^1)\right\}=3 \max\left\{\ell(\cl A),\,\ell(\cl B)\right\},$$ and so $\cl A\tens\limits_{min} \cl B$ satisfies (SP) \cite{pisier3}.
\end{proof}

\begin{ex}
 For any Hilbert space $H$ and any discrete amenable group $G$ the $C^*$-algebra $\cl B(H)\tens\limits_{min} C^*(G)$ satisfies (SP).
Let $\cl A$ be a $C^*$-algebra satisfying (SP) and $n\in\bb N.$ Since $M_n(\cl A)\cong \cl A\tens\limits_{min} M_n(\bb C)$ and $M_n(\bb C)$ is nuclear, Proposition \ref{gen} implies that $M_n(\cl A)$ satisfies (SP). In the following Proposition we prove the converse direction.
\end{ex}

\begin{cor}
\label{fd}
 Let $\cl A$ be a $C^*$-algebra and $n\in\bb N.$ If $M_n(\cl A)$ satisfies (SP) then $\cl A$ satisfies (SP).
\end{cor}

\begin{proof}
  We assume that $M_n(\cl A)$ satisfies (SP) and let $\rho\colon \cl A\to \cl B(H),$ where $H$ is a Hilbert space, be a bounded homomorphism. We consider the $n$-th amplification $$\rho_n\colon M_n(\cl A)\to \cl B(H^n),\,\,\rho_n([A_{ij}])=[\rho(A_{ij})]$$ which is a bounded homomorphism of $M_n(\cl A)$ into $\cl B(H^n).$ Since $M_n(\cl A)$ satisfies (SP), by \cite[Theorem 1.10]{haag} the map $\rho_n$ is completely bounded, and thus $\rho$ is completely bounded, as desired.
\end{proof}

We conclude this section with two results on type $II_1$ factors.

\begin{prop}
    \label{inj}
 Let $\cl M$ be a type $II_1$ factor and $n\in\bb N.$ Then $\cl M$ satisfies (WSP) if and only if $M_n(\cl M)$ satisfies (WSP).   
\end{prop}

\begin{proof}
     We assume that $M_n(\cl M)$ satisfies (WSP) and let $u\colon \cl M\to \cl B(H)$ be a $\rm{w^*}$-continuous, unital and bounded homomorphism. The $n$-th amplification $$u_n\colon M_n(\cl M)\to \cl B(H^n),\,\,u_n([X_{ij}])=[u(X_{ij})]$$ is a $\rm{w}^*$-continous, unital and bounded homomorphism of $M_n(\cl M)$ into $\cl B(H^n).$ By the fact that $M_n(\cl M)$ satisfies (WSP) it follows that $u_n$ is similar to a $*$-homomorphism, and thus $u_n$ is completely bounded, \cite[Theorem 1.10]{haag}. Thus $u$ is completely bounded. For the converse direction we use Corollary 4.7 in \cite{ele-pap}.
\end{proof}


\begin{cor}
     Let $\cl M$ be a type $II_1$ factor and $\tau$ its normal tracial state. Let also $n\in\bb N$ and $P$ be a projection in $\cl M$ such that $\tau(P)=\frac{1}{n}.$  Then $\cl M$ satisfies (WSP) if and only if $P \cl M P$ satisfies (WSP).
\end{cor}

\begin{proof}
     By \cite[Proposition 4.2.5]{anpop} the von Neumann algebras $\cl M$ and $M_n(P \cl M P)$ are $*$-isomorphic. Therefore, $\cl M$ satisfies (WSP) if and only if $M_n(P \cl M P)$ satisfies (WSP). According to Proposition \ref{inj}, $\cl M$ satisfies (WSP) if and only if $P \cl M P$ satisfies (WSP). 
\end{proof}

Other results on similarities for tensor products of certain $C^*$-algebras as well as on similarities for tensor products of type $II_1$ factors exist in \cite{pop, pop1}.
\vspace{0.5 em}
\begin{flushleft}
  {\bf Acknowledgements:} I am grateful to Iason Moutzouris for his helpful suggestions during the preparation of this paper.  
\end{flushleft}


\noindent

\begin{thebibliography}{99}

\bibitem{anpop}
Anantharaman, C., Popa, S.: {\em An introduction to $II_1$ factors}, preprint {\bf 8} (2017)




\bibitem{bm}
Blecher, D. P., C. Le Merdy.:
\textit{Operator algebras and their modules -- An operator space approach},
London Mathematical Society Monographs, New Series, \textbf{30}, The Clarendon Press, Oxford, 2004.

\bibitem{bo}
Brown, N. P., Ozawa, N.: {\em $\textrm{C}^*$-Algebras and Finite-Dimensional Approximations}, American Mathematical Soc. {\bf Vol 88} (2008)

\bibitem{ef-choi}
Choi, M. D., Effros, E. G.: {\em Nuclear C*-algebras and the approximation property}, American Journal of Mathematics {\bf 100(1)} (1978), 61-79.


\bibitem{chris1}
Christensen, E.: {\em Finite von Neumann algebra factors with property $\Gamma$}, J. Funct. Anal.
 {\bf 186} (2001), 366-380.






\bibitem{ef-lan}

Effros, E. G., Lance, E. C.: {\em Tensor products of operator algebras}, Advances in Mathematics {\bf 25(1)} (1977), 1-34.

\bibitem{ele-pap}
Eleftherakis, G. K., Papapetros, E.: {\em The similarity problem and hyperreflexivity of von Neumann algebras}. \href{https://arxiv.org/pdf/2306.07605.pdf}{arXiv:2306.07605} 






\bibitem{haag} Haagerup, U.: {\em Solution of the similarity problem for cyclic representations of $C^*$-algebras},
Annals of Math. {\bf 118} (1983), 215-240.





\bibitem{kad} Kadison, R.: {\em On the orthogonalization of operator representations}, Amer. J. Math 
{\bf 77} (1955), 600-622.




\bibitem{la}
Lance, C.: {\em On nuclear C*-algebras}, J. Funct. Anal. {\bf 12(2)} (1973), 157-176.



\bibitem{pap}
Papapetros, E.: {\em A new approach to the similarity problem}, Advances in Operator Theory {\bf 9(3)} (2024): 65.
 

\bibitem{pop}
Pop, F.: {\em The similarity problem for tensor products of certain C*-algebras}, Bulletin of the Australian Mathematical Society {\bf 70(3)} (2004), 385-389.

\bibitem{pop1}
Pop, F.: {\em Similarities of Tensor Products of Type $II_1$ Factors}, Integral Equations and Operator Theory {\bf 89} (2017), 455-463.

\bibitem{length}
Pisier, G.: {\em Joint similarity problems and the generation of operator algebras with bounded length}, Integral Equations and Operator Theory {\bf 31} (1998), 353-370.

\bibitem{pisier3} Pisier, G.: {\em The similarity degree of an operator algebra}, St. Petersburg Math. J. 
 {\bf 10} (1999), 103-146.

\bibitem{pis}
Pisier, G.:  {\em Remarks on the similarity degree of an operator algebra}, International Journal of Mathematics {\bf 12(04)} (2001), 403-414.


\bibitem{pisier1} 
Pisier, G.: {\em Similarity problems and completely bounded maps} , Lecture 
notes in Mathematics, {\bf V0l. 1618}, Springer Verlag, Berlin, (2001).



\bibitem{pisier4}
Pisier, G.: {\em Similarity problems and length}, Taiwanese J. Math. {\bf 5} (2001), 1-17

 \bibitem{pisbook}
Pisier, G.:
 {\em  Introduction to operator space theory}, 
Cambridge University Press, (2003).


\bibitem{pn}
Pisier, G.: {\em A similarity degree characterization of nuclear $C^*$-algebras}, Publications of the Research Institute for Mathematical Sciences {\bf 42(3)} (2006), 691-704.

\bibitem{pairs}
Pisier, G.: {\em Simultaneous similarity, bounded generation and amenability}, Tohoku Math. J. Second Series {\bf 59(1)} (2007), 79-99.


\bibitem{pistensor}
Pisier, G.:{\em Tensor Products of $C^*$-Algebras and Operator Spaces: The Connes–Kirchberg Problem} {\bf (Vol. 96)} (2020), Cambridge University Press.






\end{thebibliography}
\end{document}